\documentclass[11pt]{article}
\usepackage{anysize}
\usepackage{amsthm}
\usepackage{amsmath}
\usepackage{amssymb}
\makeatletter
\newcommand{\removelatexerror}{\let\@latex@error\@gobble}
\makeatother
\usepackage{bbm}

\newcommand{\calH}{{\mathcal{H}}}

\newtheorem{theorem}{Theorem}

\newtheorem{lemma}{Lemma}

\newtheorem{corollary}{Corollary}

\usepackage{thmtools}
\usepackage{thm-restate}

\usepackage{hyperref}

\usepackage{cleveref}

\begin{document}
\title{How many matchings cover the nodes of a graph?}
\author{Dehia Ait Ferhat\thanks{Mentor Graphics, Montbonnot Saint Martin,
    France},  Zolt\'an Kir\'aly\thanks{E\"otv\"os University, Budapest,
    Department of Computer Science, partially supported by a grant (no.\ K
    109240) from the National Development Agency
    of Hungary, based on a source from the Research and Technology Innovation
    Fund}, Andr\'as Seb\H{o}\thanks{CNRS,  G-SCOP, Univ.~Grenoble Alpes, partially supported  by IDEX-IRS SACRE},
	Gautier Stauffer\thanks{Center of Excellence in Supply Chain Innovation and Transportation, Kedge Business School, Talence, supported by Mentor Graphics}}

\maketitle 

\begin{abstract}
Given an undirected graph, are there $k$ matchings whose union covers all of
its nodes, that is, a {\em matching-$k$-cover}? A first, easy polynomial
solution from matroid union is possible, as already observed by Wang, Song and
Yuan (Mathematical Programming, 2014). 
However, it was not satisfactory neither from the algorithmic viewpoint nor
for proving graphic theorems, since the corresponding matroid ignores the
edges of the graph.  

We prove here, simply and algorithmically:  {\em all nodes of a graph can be
  covered with $k\ge 2$ matchings if and only if for every stable set $S$ we
  have $|S|\le k\cdot|N(S)|$.} When $k=1$, an exception occurs: this condition
is not enough to guarantee the existence of a matching-$1$-cover, that is, the
existence of a perfect matching, in this case Tutte's famous matching theorem
(J. London Math. Soc., 1947) provides the right `good' characterization. The
condition above then guarantees only that a perfect $2$-matching exists, as
known from another theorem of Tutte (Proc. Amer. Math. Soc., 1953).   

Some results are then deduced as consequences with surprisingly simple proofs,
using only the level of difficulty of bipartite matchings. We give some
generalizations, as well as a solution for minimization if the edge-weights
are non-negative, while the edge-cardinality maximization of
matching-$2$-covers turns out to be already NP-hard.  
       
We have arrived at this problem as the line graph special case of a model
arising for manufacturing integrated circuits with the technology called
`Directed Self Assembly'.
   \end{abstract}


\section{Introduction}\label{sec:introduction}

In this paper, we consider simple graphs, that is, graphs with no loops or parallel edges. We start with a few definitions and notations. 

We denote by $G=(V,E)$ a graph with {\em node set} $V$ and {\em edge set} $E$.  When no confusion may arise, we denote by $n$ and $m$ the number of nodes in $V$ and the number of edges in $E$ respectively.  

Given a graph $G=(V,E)$, an edge set $F\subseteq E$ {\em covers} a node $v$ if
there is an edge of  $F$ incident to $v$.  A family of edge sets {\em covers}
the nodes covered by its union.  A {\em matching} is a pairwise node-disjoint
subset of edges; a matching is {\em perfect} if it covers $V.$  A {\em
  $2$-matching} is a collection of node-disjoint edges and circuits;  a
2-matching is {\em perfect} if it covers $V.$ A {\em matching-$k$-cover}
$(k\in\mathbb{N})$ is a set of $k$ matchings that cover $V.$ The union of (the
edges of) these matchings will also be called a matching-$k$-cover. The {\em
  cardinality} or {\em weight} (in edge-weighted graphs) of a
matching-$k$-cover is the cardinality or weight of its edge-set.  We denote by $\nu(G)$ the {\em matching number}, that is, the size of a maximum matching of the graph $G$.  

Given $U \subseteq V,$ we denote by $N(U)$ the set of neighbors of nodes in
$U$, that is, $N(U):=\{v\in V \setminus U: \hbox{ there exists } u\in U, uv\in
E \}$. We  let $\delta(U)$ denote the set of edges with exactly one extremity
in $U$; $d(v)=|\delta(\{v\})|$ is the {\em degree} of node $v\in V$;
$\Delta:=\max_{v\in V}d(v)$ is the maximum degree. For $U\subseteq V$, $G[U]$
denotes the {\em graph induced by $U$}, with node set $U$, and all edges of
$G$ with both end nodes in $U$.  A {\em path} is a connected subgraph with all
degrees at most two, and with at least one node of degree less than two.
The nodes of degree one, if any, are its {\em end nodes}.  A  {\em circuit} is a connected subgraph with all degrees exactly two.

A {\em $k$-star graph}, consisting of a stable set $S$ of size $k$ and an
additional node joined to all nodes of $S$, shows that the minimum number of
matchings that cover all nodes can be as large as the maximum degree. A na\"\i
ve answer to the title question could then be nourished by Vizing's theorem
\cite{Vizing64} stating that {\em all edges} of a graph can be covered by at
most $\Delta+1$ matchings.

Actually, K\H{o}nig's ``edge-coloring theorem'' \cite{Konig,KonigNemetAngol}
already implies
that $\Delta$ matchings are always enough to cover all nodes at least once
(since a spanning tree is a bipartite graph). This is best possible in
terms of the maximum degree.  However, obviously there
are graphs for which a smaller number of matchings than the maximum degree is
sufficient (eg.\ $K_4$, the complete graph on four vertices).

Now, for any stable set $S$ (if $N(S)\ne\emptyset$), $\frac{|S|}{|N(S)|}$ is a lower bound on the number of matchings needed to cover the vertex set as shown by Lemma \ref{lem:necessary}.

\begin{lemma}\label{lem:necessary}
  If $G$ is a connected graph with $n>1$ vertices and $S\subseteq V$ is a non-empty stable set, then at least $\frac{|S|}{|N(S)|}$ matchings are needed to cover the vertex set.
\end{lemma}

\begin{proof}
  If a matching-$k$-cover exists, then   $\frac{|S|}{|N(S)|}\le k$, because one
  matching can cover at most $|N(S)|$ nodes of the stable set $S$.
\end{proof}

 A key result of the paper is that {the maximum of these ratios is equal to the minimum number of matchings in a cover, provided that the maximum is strictly greater than one}: 
 
\begin{restatable}{corollary}{maincorollary}
\label{cor:main}
  The minimum number of matchings that  cover all nodes of a graph is equal to the maximum of $\left\lceil \frac{|S|}{|N(S)|}\right\rceil$ over non-empty stable sets of the graph, unless this maximum is  $1$, in which case it is 1 or 2 according to whether  a perfect matching exists or not.
\end{restatable}

For an example, on the $k$-star graph the maximum is $k$ and indeed, the minimum number of matchings covering $V$ is also $k$.\\

The condition of this theorem  in the exceptional case $k=1$ reveals a remarkable connection and clarifies the exception. The condition coincides then with the following necessary and sufficient condition of Tutte for the existence of $2$-matchings:  

\begin{theorem}[Tutte, 1953, \cite{Tutte53}]\label{thm:2m}
	Let $G=(V,E)$ be a graph. There exists a perfect $2$-matching in $G$ if and only if  $|S|\le |N(S)|$ for each stable set $S$.
\end{theorem}

\newpage
\subsection*{Summary of our contributions and organisation of the paper}

The minimum value $k$ where a matching-$k$-cover exists was calculated by Wang, Song and Yuan in time $O(nm)$. This algorithm was published in  \cite{erratum}
where the authors corrected their previous paper \cite{matching_cover_china}. The authors also observed that the problem can be solved using matroid union (of matching matroids) and hence the node-weighted version of the problem can also be solved in polynomial time (note that our new minmax result (Corollary \ref{cor:main}) does not follow from the minmax theorem from matroid union). The algorithm of Wang, Song and Yuan \cite{erratum} builds upon Edmonds Gallai theorem for non bipartite matching plus some augmenting path techniques and some additional technical arguments. 

In this paper, we show that the problem is actually much simpler than what it looks like at first sight by identifying connections with other well-known results from the area. We use this to develop a simple augmenting path algorithm (and a short proof) in the spirit of bipartite matching (no need for Edmonds- Gallai decomposition theorem). Besides giving a simplified algorithmic result, we also develop a good characterization (that is, the minmax result described in Corollary \ref{cor:main}, again in the spirit of bipartite matching). Then we extend the result to deal with the weighted case and some other extensions and make clear connections to other well-known problems in the area on {\it star-packings}  and {\it (1,k)-factors} (definitions in Section \ref{sec:conn})  by Tutte \cite{Tutte54}, Berge and Las Vergnas \cite{BergeLasvergnas}, Vizing \cite{Vizing64} Hell and Kirkpatrick  \cite{HellKirkpatrickAlg} (simplifying some of the corresponding proofs at the same time, and establishing the exact relations between these independent results that have been ignoring one another). All in all, the understanding of this fundamental problem is substantially simplified and extended.

In Section~\ref{sec:ex} we state and prove the key result of the paper in a simple form, providing a ``good characterization theorem'' for the existence of matching-$k$-covers. 
In Section~\ref{sec:conn} we clarify connections with two other
combinatorial objects, connecting matching-$k$-covers to previously studied
notions, thereby enriching  the available tools. As an immediate consequence the
polynomial decidability of the existence, and {\em solvability of  weighted
  versions} (for non-negative weights) is deduced.      In Section~\ref{sec:variants} the results and tools are used to widen the context of our work. Finally, in Section~\ref{sec:app} we describe the practical problem that has been the source of our study.

\section{Existence of a matching-$k$-cover}\label{sec:ex} 

In this Section we prove the key existence result concerning matching-$k$-covers, that is,  Theorem \ref{thm:main}. At the technical level, this is almost all we need. The rest of the paper consists in realizing connections and consequences, then considering weighted versions and generalizations.

We keep in this section the somewhat na\"\i ve state of ignoring the connections that we are discussing (and have realized) later. These connections are discussed in the next section. The goal here is to show  a self-contained first proof for matching-$k$-covers without these connections. 

\begin{theorem}
\label{thm:main}
 Let $G=(V,E)$ be a connected graph, and $k\geq 2$ be an integer. There exists a matching-$k$-cover in $G$ if and only if  $|S|\leq k\,|N(S)|$ for each stable set $S$ of $G$.
\end{theorem}

\begin{proof} We have already checked the only if part in Lemma \ref{lem:necessary}. 

\medskip	
\noindent{\bf Claim 1:} If $F\subseteq E$ and $d_{F}(v)\le k$ for every $v\in V,$ then $F$ contains $k$ matchings that cover the same nodes as $F$.   
	
\smallskip

While there is an edge  $uv\in F$ such that both $u$ and $v$ are  covered by at least two edges of $F$, delete $uv$ from $F$.  
At the end of the procedure, the connected components of $F$ are stars of
degree at most $k$. We can thus build $k$ matchings covering the same node set as $F$ by assigning the edges in each star to a different matching.  
	
\smallskip
	 
In order to prove the if part of the theorem, let $F$ be a union of $k$
matchings that covers a maximum number of nodes. We choose $F$ to be
inclusionwise minimal among these maximum choices. 
	
Let  $B$  be the set of nodes incident to at most one edge of $F$, let $A:= V \setminus B$ be the rest of the nodes of $G$, and let $B'\subseteq B$ be the nodes of $V$ not covered by  $F$. By the minimality of $F$, and by Claim 1, no edge of $F$ connects two nodes of $A$. We show that $B'\ne \emptyset$ implies that there exists a stable set $S$ such that $|S|>k\cdot |N(S)|$. Let us assume that $B'\neq \emptyset$.
	
	Let $A'$  be the set of nodes of $A$ reachable from  $B'$ by an {\em alternating path}, that is, a path with nodes alternating between $B$ and $A$, and edges alternating between  $E\setminus F$ and $F$. Let  
	$S\subseteq B$ the set of nodes of $B$ reachable from $B'$ by an alternating path, clearly $S\supseteq B'$.

	\medskip\noindent
	\noindent{\bf Claim 2:}
	{\em The nodes in $A'$ are incident to exactly  $k$ edges of $F$.} 
	
	\smallskip
	Indeed, suppose there is a node in $A'$ incident to less than $k$ edges of $F$. This node is reachable from $B'$ by an alternating path by construction. Interchanging edges of $F$ and $E\setminus F$ along this alternating path yields a set $F'$ with $d_{F'}\leq k$ covering more nodes of $G$ (because a node of $B'$ that was not covered is now covered), and hence, by Claim~1, there exists a set of $k$ matchings that cover more nodes than $F$, a contradiction.    
	
	\medskip\noindent
	\noindent{\bf Claim 3:} {\em $S$ is a stable set.}
	
	\smallskip
	 We prove that there is no edge from $S$ to any node of $B$,
         which implies in particular that $S$ is a stable set and that $N(S)
         \subseteq A$. Indeed, assume to the contrary that there is an edge
         $(s,b)$ for some $s\in S$ and $b\in B$. By definition, there is an
         alternating path $P$ from some $b'\in B'$ to $s$, and since the last
         edge of $P$ is in $F$ (or $s\in B', \; P=\{s\}$), $(s,b)\notin F.$
         But then $P$ could be extended with $(s,b)$ to an alternating path
         (not completely alternating between $A$ and $B$, since its last two
         nodes, $s$ and $b$ are both in $B$, which is not disturbing) from
         $b'$ to $b$. Interchanging edges of $F$ and $E\setminus F$ along this
         alternating path yields a set $F'$ with $d_{F'}\leq k$ (because
         $d_F(b)\le 1<k$), covering more nodes of $G$, and hence, by Claim~1,
         there exists a set of $k$ matchings that cover more nodes than $F$, a
         contradiction.

	\medskip\noindent
	\noindent{\bf Claim 4:} {$N(S)=A'.$}
	
	\smallskip

        Each vertex of $A'$ was marked from a vertex of $S$ and hence
        $A'\subseteq N(S)$. Moreover, there is no edge from any $s\in S$
        to a node  $a\in A\setminus A'$ since the existence of such an edge in
        $E\setminus F$ would prove $a\in A'$ (and by the definition of $B$
        if $s$ is incident to an edge in $F$, then there is only one such
        edge and it was used as the last edge of the alternating path reaching
        $s$, i.e., it comes from $A'$). The claim follows.

	\smallskip
	By Claim~2,  Claim~3 and Claim~4,  $|S\setminus B'| = k\,|A'|$, and since  $B'\ne\emptyset$, we have $|S|> k\,|N(S)|$.
\end{proof}

At first sight, it is surprising that the condition of Theorem~\ref{thm:main} is the same as that of a Theorem of Berge and Las Vergnas \cite{BergeLasvergnas}, which concerns seemingly different combinatorial objects. The proof also revealed connections to Lov\'asz and Plummer's Exercise \cite[Exercise 10.2.28]{matching_theory} and in turn to Heinrich, Hell, Kirkpatrick, Liu \cite{Hell}. We will clarify these connections in Section \ref{sec:conn}. We can readily deduce from Theorem~\ref{thm:main} the main result claimed in the introduction and a  pleasant minimax formula for the maximum number of nodes covered by $k$ matchings.

\maincorollary*

\begin{corollary}
  Given $G$ and $k\ge 2$, the maximum number of nodes that can be covered by $k$ matchings is $|V|-\max\{|S|-k\,|N(S)| \;:\; S$ is a stable set$\}$.
\end{corollary}

\begin{proof}
  It is obvious that we cannot cover more. In the proof of the theorem above
  the number of uncovered nodes is
  $|B'|=|S|-k\,|A'|=|S|-k\,|N(S)|$.
\end{proof}

Theorem~\ref{thm:main} can actually also be extended to the case where we wish to cover only a subset of the vertices. This is the purpose of the next theorem.

\begin{theorem}\label{thm:mainU}
  Given a graph $G$, an integer $k\ge 2$ and $U\subseteq V,$ there are $k$ matchings whose union covers $U$ if and only if for each stable set $S\subseteq U$, we have $|S|\le k\,|N(S)|$. 
\end{theorem}

\begin{proof}
Let $F$ be a union of $k$
matchings that covers a maximum number of nodes of $U$; furthermore, let $F$ be
inclusionwise minimal among these maximum choices.
Let $A$ and $B$ the  sets defined in the proof of Theorem \ref{thm:main} but now let  $B'\subseteq B\cap U$ the nodes of $U$ not covered by  $F$.
We can repeat the previous proof.
However, Claim~3  should be replaced as follows.

	\medskip\noindent
	\noindent{\bf Claim 3'}: {\em $S$ is a stable set and  $S\subseteq U$.}
	
	\smallskip
        First suppose that there is an $s\in S\setminus U$. There is an alternating path from some $b\in B'$ to $s$  by the definition of $S$, after  interchanging along such a path $b\in U$ becomes covered (and $s\not\in U$ becomes uncovered), a contradiction.

        From this point we can follow the original proof.
\end{proof}

\begin{corollary}
  Given graph $G$, integer $k\ge 2$ and $U\subseteq V,$ the maximum number of nodes in $U$ that can be covered by $k$ matchings is $|U|-\max\{|S|-k\,|N(S)| \;:\; S\subseteq U$ is a stable set$\}$.
\end{corollary}

\begin{proof}
  It is obvious that we cannot cover more. In the proof of the theorem above
  the number of uncovered $U$-nodes is $|B'|=|S|-k\,|A'|=|S|-k\,|N(S)|$.
\end{proof}

\bigskip

The proofs of Theorem \ref{thm:main} and Theorem \ref{thm:mainU} can be straightforwardly turned to algorithms. Consider the proof of Theorem  \ref{thm:main} for instance.  If we orient the edges of $F$ from $A$ to $B$ and the edges of $E\setminus F$ from $B$ to $A$, then determining $A'$ and $S$ is an accessibility problem (they are the nodes accessible from $B'$), so they can be found by Breadth First Search in time $O(m)$.  Either we can do the interchange and  increase the number of nodes that are covered or we find a stable set $S$ presenting the obstacle. After increasing the number of covered nodes, we may need to delete some edges of $F$ in order to make it inclusionwise minimal, and we may need to redefine sets $A, B, B'$. These steps can easily be done in time $O(m)$. Decomposing $F$ into $k$ matchings can also be done in time $O(m)$. Therefore the proof readily provides an algorithm which correctly determines either a matching-$k$-cover or an obstacle $S$ in time $O(nm)$. 

Suppose now that the goal is to find the smallest $k\geq 2$ for which a
matching-$k$-cover exists (ignoring the distinction between the cases $k=1$
and $k=2$). Starting with $k=2$ we execute the previous algorithm, and if an
obstacle is found, then we increase $k$ by one and start it again. If we
implement the corresponding algorithm na\"ively the running time is
$O(mn^2)$. However we can use the $k$ matchings from the previous iteration to
warm-start the next one. Either the number of covered nodes  or $k$ increases
in each step, and the complexity of the corresponding algorithm then remains
$O(mn)$, matching the time bound of \cite{erratum}.

There is another, conceptually more complicated algorithm of Hell and Kirkpatrick \cite{HellKirkpatrickAlg} for the related $(1,k)$-factor problem that runs in time $O(m\sqrt{n})$ if $k\ge 2$. We will see in the next section that this algorithm can also be used to decide whether a matching-$k$-cover exists or not. Using binary search for $k$ we need to run it for at most $\log n$ times, consequently the minimum $k$, for which a matching-$k$-cover exists, can be be found in time $O(m\sqrt{n}\log n)$ if the answer is at least two. 

Now suppose that we {\em have to make a distinction} between the cases $k=1$ and $k=2$. If the algorithm finds a matching-$2$-cover, then we need an algorithm for finding a perfect matching (in this case $k=1$ is the right answer) or an obstacle (when $k=2$ is the right answer).
Although this is technically much more difficult, there are several $O(m\sqrt{n}\log n)$ time algorithms for that, see e.g., Section 24.4 in \cite{Schrijver}. Therefore the unconstrained minimum of $k$  can be determined in $O(m\sqrt{n}\log n)$ time.

 \section{Connections}\label{sec:conn}
 
In this section we define two other notions seemingly different from one another and from matching-$k$-covers. We reveal the somewhat surprising but simple, and at the same time, fruitful connection of the three, and also the limits of this connection.

 A {\em $k$-star packing} is a (not necessarily spanning) subgraph whose components are  $k'$-star graphs with 
 $1\le k'\le k$. If a $k$-star packing covers every node it is called {\em perfect}.   A {\em $(1,k)$-factor} $(k\in \mathbb{N})$  is an edge set  where every $v\in V$ has at least $1$ and at most $k$ incident edges.  The {\em cardinality} of a matching-k-cover, a $k$-star packing or a $(1,k)$-factor is the number of edges in the corresponding set.

We will realize now that, as long as we are interested in minimizing the cardinality (or more generally any non-negative cost function on the edge set), the perfect $k$-star packing problem, the $(1,k)$-factor problem, and the matching-$k$-cover problem coincide. In contrast, if we have negative costs on some edges, each problem might behave very differently.

Let us first notice the {\em obvious containments between three corresponding sets}: each perfect $k$-star packing is clearly a matching-$k$-cover,  and each matching-$k$-cover is obviously a $(1,k)$-factor. None of these containments is reversible though: two node-disjoint $k-1$-stars joined by an edge is a matching-$k$-cover on $n=2k$ nodes, but not a $k$-star packing (even though it contains a perfect $k$-star packing); and an odd circuit is a $(1,2)$-factor, but not a matching-$2$-cover (even if it contains one). However, their {\em inclusionwise minimal} elements are the same:

\begin{theorem}\label{thm:equivproblems}
Let $G=(V,E)$ be a graph, $k\in\mathbb{N}$ and $F\subseteq  E$. Then $F$ is an inclusionwise minimal perfect $k$-star packing if and only if it is an inclusionwise minimal matching-$k$-cover, which in turn holds  if and only if it is an inclusionwise minimal $(1,k)$-factor. Consequently, a minimum weight matching-$k$-cover can be found in polynomial time for non-negative weights. 
\end{theorem}

	\begin{proof}  Because  of the aforementioned containments between the three sets, it is sufficient to prove that each inclusionwise minimal $(1,k)$-factor is a  perfect $k$-star packing. This is clear: while there exists an edge between two nodes both incident to more than one edge in a $(1,k)$-factor, one can delete this edge, and we get a $(1,k)$-factor again.  A $(1,k)$-factor  without such an edge is a perfect $k$-star packing. 
		
	   Since the minimum of non-negative weight functions on non-negative
           vectors is always attained on vectors of inclusionwise minimal
           support, it is sufficient to check that one of the three problems
           can be solved in polynomial time for non-negative weights. This is
           well-known for $(1,k)$-factors. Indeed, a $(1,k)$-factor of minimum
           weight can be found in (strongly) polynomial time for any objective
           function for instance by the reductions of Tutte \cite{Tutte54} and
           Edmonds's minimum cost matching algorithm \cite{Edmonds-weighted,
           EJ}, see also \cite[35.2, p.\ 586]{Schrijver} or \cite{Bert}.
\end{proof}

Before concluding erroneously about the total confusability of the three defined notions, let us check now how they behave with respect to objective functions having also negative coordinates. The task of minimizing an arbitrary cost function on a family of sets  includes maximizing the cardinality of the sets in the family.  The complexities of the three corresponding cardinality maximization problems are already quite different:  
	\begin{itemize}
		\item[--] Maximizing the cardinality of a matching-$k$-cover is NP-hard already for $k=2$, since finding a matching-$2$-cover of size $n$ in a cubic graph is equivalent to the existence of two edge-disjoint perfect matchings, which is the three-edge-coloring problem of cubic graphs, a well-known NP-complete problem \cite{Holyer}. 
	        \item[--] Maximizing the cardinality of a $3$-star packing is NP-hard: there exists a $3$-star packing of  cardinality at least $3n/4$ if and only if there exists a ``perfect 3-star packing''. This problem has been proved to be NP-complete by Hell and Kirkpatrick \cite{HellKirkpatrick}.
	\item[--] For any cost function maximizing the cost of a
          $(1,k)$-factor is polynomially solvable, see e.g.,  \cite[35.2, p.\ 586]{Schrijver}.
\end{itemize}

 \section{Variants}\label{sec:variants}

 In this section we state several variants and generalizations of Theorem~\ref{thm:main}.

We first consider the case where we need to cover the nodes of $G$ more than
once. Suppose lower bounds $\ell(v)\in \mathbb{N}$ are given on vertices.  A
matching-$k$-cover is $\ell$-bounded if every vertex $v\in V$ is covered by at
least $\ell(v)$ pairwise edge-disjoint matchings, such that, if
$\mathcal{M}=\cup_{i=1}^k M_i$ denotes the union of the $k$ matchings, then
$d_\mathcal{M}(v)\ge \ell(v)$ for each $v\in V$.

Again, if  $k=\ell(v)$ is allowed, the problem has a particular nature
different from the restricted case where $\ell(v)<k$ for every $v$. While for
$k=\ell=1$ the solution is   polynomially solvable, even if harder than for
$k>\ell=1$, the $k=\ell=2$ case is already NP-hard: defining $\ell(v)=2$ for
all $v\in V$ in a $3$-regular graph, an  $\ell$-bounded matching-$2$-cover exists if and only if the graph is $3$-edge-colorable, which is NP-hard to decide \cite{Holyer}.

However, our results can be generalized without any major change if $\ell(v)<k$
for every $v\in V$, despite an additional difficulty that can be overcome with
invoking another piece of the literature.

Recall the ultimate explanation given by Theorem~\ref{thm:equivproblems} for the tractability of matching-$k$-covers. 

{\em Inclusionwise minimal $(1,k)$-factors are $k$-star packings, and thus matching $k$-covers}. 

The last implication was not explicitly explained, since it is obvious: a $k$-star is $k$-edge-colorable, and so are the edges of a $k$-star packing, and this is why  the set of these edges is a matching-$k$-cover. 

Is the same true for {\em $(\ell,k)$-factors}, that is, subgraphs $F$ with
$\ell(v)\le d_F(v)\le k$ for every $v\in V$?  The answer is yes whenever
$\ell(v)<k$  but it is no more trivial, our proof uses a less known theorem
about edge-coloring:  

\begin{lemma}\label{lem:Fournier}
Let $k>\ell(v)$ for all $v\in V.$ An inclusionwise minimal $(\ell,k)$-factor
is  an $\ell$-bounded matching-$k$-cover.
\end{lemma}

\begin{proof}
Let $F$ be an inclusionwise minimal $(\ell,k)$-factor. Since $F$ is inclusionwise minimal, {\em and $k>\ell$}, no edge of 
$F$ connects $u$ and $v$ if $d_F(u)=d_F(v)=k$.  By Fournier's \cite{Fournier} generalization of Vizing's \cite{Vizing64} theorem, in this case $F$ is $k$-edge-colorable.
\end{proof}

It follows that our results can be generalized to  $\ell$-bounded matching-$k$-covers: 

For  {\em non-negative} weights a polynomial algorithm directly follows from   Lemma~\ref{lem:Fournier}, since we know this for  $(\ell,k)$-factors.  For the cardinality case the direct proofs and algorithms of Section~\ref{sec:ex} arrive at the following theorem. By Lemma~\ref{lem:Fournier} this theorem is equivalent to 
the theorem of  Heinrich, Hell, Kirkpatrick and Liu \cite{Hell} stating that the same condition  is necessary and sufficient for the existence of an $(\ell,k)$-factor. 

\begin{theorem}\label{thm:l-bounded} 
	Let $G$ be a graph, $k\ge 2$ and $\ell: V\to \{0,\dots,k-1\}$. There exists an $\ell$-bounded matching-$k$-cover in $G$ if and only if  $\sum_{x\in V-X} \max(\ell(x)-d_{G-X}(x),0)\le k|X|$ for every $X\subseteq V.$
\end{theorem}

\bigskip
Let us generalize now the problem to an arbitrary family $\cal H$ of subgraphs  allowed to be used for covering the vertices of $G$.

 Let $\cal H$ be a family of subgraphs of $G$. An $\cal H$-$k$-cover consists of $k$ elements of $\cal H$ that altogether cover every node of $G$.
If $\cal H$ is the family of all matchings of $G$, then we get back  matching-$k$-covers. What happens if we allow some richer family that can be used in the cover? We are going to see a few examples with interesting answers closely related to matching-$k$-covers. 

First we examine the case where $\calH$ is the family of $2$-matchings. 
Interestingly enough, we get the same answer as for matching-$k$-covers, but $k=1$ is not an exception anymore.

\begin{theorem}\label{thm:2m-k-cover} 
Let $G=(V,E)$ be a graph, and $k\ge 1$ be an integer. There exists a $2$-matching-$k$-cover in $G$ if and only if $|S|\leq k\,|N(S)|$ for each stable set $S$ of $G$.
\end{theorem}

\begin{proof} 
  For $k=1$ this is Theorem \ref{thm:2m}.
  
  For $k\ge 2$ it is enough to prove that an edge-minimal $2$-matching-$k$-cover is also a matching-$k$-cover. 

Let $H_1,\ldots,H_k\in\calH$ be the $2$-matchings in an edge-minimal cover. First observe that  $C$ has no even circuit component: such a component could be replaced by any  of its perfect matchings, replacing the edge-set by a proper subset, a contradiction. 

Suppose now that $C$ is an odd circuit component of some $H_i$. If $C$ has a node $v$ covered by some $H_j$, $(j\ne i)$,  then replace $C$ in $H_i$ by a perfect matching of $C-v$: the set of covered nodes does not change, but the set of edges is replaced by a proper subset, a contradiction again. 
  
Finally, if  $C$ is an odd circuit node-disjoint from every other $H_j$, then we express it as the union of two (not node disjoint) matchings $M_1, M_2$ covering all nodes of $C$, replace $C$ by $M_1$ in $H_i$ and add $M_2$ to any  $H_j \; (i\ne j)$. This is possible, since $k\ge 2$. 
\end{proof}

We can also handle the case of $2$-star-packings.

\begin{lemma}\label{lem:2st}
	Let $k\in\mathbb{N}$, and  $\calH$ be the set  of $2$-star-packings. An edge-set is an inclusionwise minimal $\calH$-$k$-cover if and only if it is a matching-$2k$-cover.
\end{lemma}

\begin{proof} 
By Theorem \ref{thm:equivproblems} the components of an edge-minimal  matching-$2k$-cover are $k'$-stars with $k'\le 2k$, and can be thus covered by at most $2k$ matchings. The other direction is obvious.
\end{proof}

Again, polynomial solvability of the minimization of non-negative weight functions and the solution of problems solvable for matching-$k$-covers can be adapted to $\calH$-$k$-covers. For instance: 
  
\begin{theorem}\label{thm:2star-k-cover} 
Let $G=(V,E)$ be a graph, $k\ge 1$ be an integer, and
  $\calH$ be the set of  $2$-star-packings. There exists an $\calH$-$k$-cover in $G$ if and only if $|S|\leq 2k\,|N(S)|$  for each stable set $S$ of $G$. 
\end{theorem}

A {\em path-packing} is a subgraph whose components are paths. Again, inclusionwise minimal path-packing-k-covers are exactly the inclusionwise minimal $2$-star-$k$- covers 

\begin{corollary}
Let $G=(V,E)$ be a graph, $k\ge 1$ be an integer, and
  $\calH$ be the set of path-packings. There exists an $\calH$-$k$-cover in $G$ if and only if $|S|\leq 2k\,|N(S)|$ for each stable set $S$ of $G$.
\end{corollary}

\begin{proof}
A $2$-star-packing is a path-packing. Given a path-packing, some edges of the paths can be deleted resulting in a $2$-star-packing.
\end{proof}

	The reader is invited to state any combination of the many possible variants of the theorems and algorithms  including $\ell$-bounds, node-weights etc.,  according to tastes and needs.

 \section{From the manufacturing of integrated circuits to  $k$-matching-cover  }\label{sec:app}

Our interest in the matching-$k$-cover problem stemmed from the study of an application in the manufacturing of integrated circuits. More specifically, we were interested in the optimization of the manufacturing of one particular type of components called vias (basically vertical connectors between different layers of an integrated circuit). 

The main process used in the manufacturing of vias is called (photo)lithography. Lithography is a technology that uses light to transfer an arrangement of (2-dimensional) geometric features from a mask to a light-sensitive chemical photoresist on the silicon wafer. The nature of the process, and diffraction in particular, imposes a minimum distance $L$ between two features in order to obtain a proper transfer. Trying to transfer simultaneously two vias (typically equal size disk features when viewed from the top) that are at a distance below $L$ would result in the production of a `dumbbell-shaped' object on the wafer.

One way of transfering an arrangement (a.k.a. layout) of vias where some are at a distance below $L$ is to decompose the layout into several subsets of vias that respect the minimum distance restriction and to transfer the corresponding sub-layout iteratively. This procedure is known as multiple patterning. Because the masks are expensive, manufacturers are interested in minimizing the number of multiple patterning steps. Decomposing the layout into a minimum number of feasible sub-layouts readily translates into a graph coloring problem: the vias are the nodes of the graph and there is an edge between two nodes if the corresponding vias are below the minimal distance.

Now a new technology, called Directed Self Assembly (DSA), offers an alternative to pure multiple patterning. The technology allows to correct `dumbbell-shaped' objects after their creation at a relatively low cost, if they respect certain properties. Actually the technology allows to correct even longer `pea pod-shaped' objects resulting from the improper transfer of a sequence of (aligned) vias that are below the minimum distance. The precise shapes of the objects that can later be corrected through DSA technology is not yet explicitly known. However there are techniques to check whether a pattern is feasible or not. 

Manufacturers want to exploit DSA technology in order to minimize the number of multiple patterning steps. Suppose that we are given a layout of vias and a complete list of the different grouping of vias that we can later correct through DSA technology. We can create a graph $G$ whose nodes are all the feasible grouping (including singletons) and where two nodes are adjacent if the (minimum) distance between the corresponding groups is below $L$ (in particular two groups containing a same via will always be in conflict). We now want to find a set of nodes $U$ in $G$ such that (i) the unions of the vias in the groups corresponding to the set $U$ covers all original vias and (ii) the chromatic number of $G[U]$ is minimum among all such $U$. Observe that the set of nodes containing a given via form a (non necessarily maximal) clique. Let us call $\cal K$ the set of all such cliques. The problem is equivalent to determine the minimum number of colors needed to partially color the nodes of $G$ so that each clique of $\cal K$ is hit by (at least) one of the colors.

We might call the underlying graph coloring problem the {\em clique-hitting coloring problem}. The problem is obviously hard as it contains (proper) graph coloring as a special case (for a given graph, choose each single node as one of the cliques of $\cal K$). Now the matching-$k$-cover appears to be a special case of this problem where the graph $G$ is the line graph of a graph $H$ and the set of cliques $\cal K$ are the cliques stemming from the $|V(H)|$ (maximal) stars in $H$ centered in each vertex (that is, ${\cal K}=\{\{\delta(v)\}$, for all $v\in V(H)\}$.

\bigskip \noindent{\bf Acknowledgment}: We thank Mentor Graphics for their role beyond the finances of a doctoral work: the problem analyzed here originates from their practical initiatives, related to new technology. 

Thanks are due to Fr\'ed\'eric Maffray for discovering that   matching-$k$-covers have  been studied in \cite{matching_cover_china}.  Our  main results seem to be complementary to theirs. Trying to understand the connections we went into the flaw of  \cite{matching_cover_china} but were reassured by the soon appearing \cite{erratum}. Our main results compared to these are the theorems that provide an appropriate framework for simple proofs and algorithm, moreover to some generalizations.

\end{document}